\documentclass[reqno,oneside,11pt]{amsart}

\usepackage[T1]{fontenc}
\usepackage{times,mathptmx}
\usepackage{amssymb,epsfig,verbatim,mathtools,float,hyperref,xcolor}
\usepackage{pgf,tikz}
\usepackage[all,cmtip]{xy}
\theoremstyle{plain}

\newtheorem{thm}{Theorem}
\newtheorem{cor}[thm]{Corollary}

\newtheorem{lem}[thm]{Lemma}

\newtheorem{proposition-principale}[thm]{Proposition principale}

\theoremstyle{definition}

\addtocounter{section}{0}              
\numberwithin{equation}{section}       
\begin{document}

\setlength{\baselineskip}{0.54cm}         
\title[{Singularities of holomorphic codimension one foliations 
of the complex projective plane}]{Singularities of holomorphic codimension one foliations \\
of the complex projective plane}
\date{}

\author{Dominique Cerveau}
\address{Univ. Rennes, CNRS, IRMAR-UMR $6625$, F-$35000$ Rennes, France}
\email{dominique.cerveau@univ-rennes1.fr}

\author{Julie D\'eserti}
\address{Universit\'e d'Orl\'eans, Institut Denis Poisson, route de Chartres, $45067$ Orl\'eans Cedex $2$, France}
\email{deserti@math.cnrs.fr}



\begin{abstract}
We prove that any holomorphic codimension $1$ foliation on the 
complex projective plane has at most one
singularity up to the action of an ad-hoc birational map. 
Consequently, any algebraic foliation on the affine plane has no
singularities up to the action of a suitable birational self map
of the complex projective plane into itself.
\end{abstract}

\maketitle

\vspace{2cm}

Let $\mathcal{F}$ be a codimension one holomorphic foliation
of degree $N$ on $\mathbb{P}^2_\mathbb{C}$. Denote by 
$\pi\colon\mathbb{C}^3\smallsetminus\{0\}\to\mathbb{P}^2_\mathbb{C}$
the canonical projection. The homogeneous foliation 
$\pi^{-1}\mathcal{F}$ extends to $\mathbb{C}^3$ and is defined
by a $1$-form 
\[
\omega=a(x,y,z)\,\mathrm{d}x+b(x,y,z)\,\mathrm{d}y+c(x,y,z)\,\mathrm{d}z
\]
where $a$, $b$, $c$ are homogeneous polynomials of degree
$N+1$ without common component satisfying the Euler
identity: $ax+by+cz=0$; it is the Chow theorem for 
foliations. The singular set $\mathrm{Sing}(\mathcal{F})$
of $\mathcal{F}$ is given by
\[
\pi(\{a=b=c=0\}\smallsetminus\{0\}).
\]
Let us recall what is the Milnor number $\mu(\mathcal{F},m)$
of a foliation $\mathcal{F}$ at a singular point $m$. Let 
us fix a local chart $(u,v)$ such that $m=(0,0)$. The 
germ of $\mathcal{F}$ at $m$ is defined, up to multiplication
by a unit at $0$, by a $1$-form $E\,\mathrm{d}u+F\,\mathrm{d}v$. 
Denote by $\langle E,\,F\rangle$ the ideal generated by 
$E$ and $F$, then 
\[
\mu(\mathcal{F},m)=\dim\frac{\mathbb{C}\{u,v\}}{\langle E,\,F\rangle};
\]
it is also the multiplicity of intersection of the germs of curves
$(E=0)$ and $(F=0)$. Let us recall that (\cite[Chapter 2, Section 3, Example 2]{Brunella})
%\cite[Theorem 2.3]{Jouanolou}
\[
\displaystyle\sum_{m\in\mathrm{Sing}(\mathcal{F})}\mu(\mathcal{F},m)=N^2+N+1;
\]
in particular, $\mathcal{F}$ has a least one singular point.

\medskip

\noindent\textbf{Example.}
Let us consider the diagonal linear foliation $\mathcal{F}$ given by the $1$-form 
\[
\omega=\lambda yz\,\mathrm{d}x+xz\,\mathrm{d}y-(1+\lambda)xy\,\mathrm{d}z
\]
with $\lambda\in\mathbb{C}\smallsetminus\{0,\,-1,\,-2\}$. Note that $\mathrm{Sing}(\mathcal{F})=\big\{(0:0:1),\,(1:0:0),\,(0:1:0)\big\}$;
moreover, $x^\lambda yz^{-(1+\lambda)}$ is a first integral of $\mathcal{F}$. 

The birational involution from $\mathbb{P}^2_\mathbb{C}$ into itself given 
by 
\[
\mathcal{I}_1\colon(x:y:z)\dashrightarrow(xy:y^2:x^2-yz),
\]
defines an isomorphism of  $\{y\not=0\}$.
Remark that the foliation $\mathcal{I}_1^{-1}\mathcal{F}$ is described by
the $1$-form
\[
-y\big((2+\lambda)x^2+\lambda yz\big)\,\mathrm{d}x+x\big((2+\lambda)x^2-yz\big)\,\mathrm{d}y+(1+\lambda)xy^2\,\mathrm{d}z;
\]
in particular, $\mathrm{Sing}(\mathcal{I}_1^{-1}\mathcal{F})=\big\{(0:0:1),\,(0:1:0)\big\}$. 

Let us now consider the birational involution $\mathcal{I}_2$ from
$\mathbb{P}^2_\mathbb{C}$ into itself given by 
\[
\mathcal{I}_2\colon(x:y:z)\dashrightarrow(x^2:-xy+z^2:xz);
\]
note that $\mathcal{I}_2$ defines an isomorphism of $\{x\not=0\}$.
Furthermore, the restriction of $\mathcal{I}_2^{-1}\mathcal{I}_1^{-1}\mathcal{F}$ on 
$(x=0)$ is described by the $1$-form $(1 - \lambda )z^5\,\mathrm{d}x$; consequently,
$\mathrm{Sing}(\mathcal{I}_2^{-1}\mathcal{I}_1^{-1}\mathcal{F})=\big\{(0:1:0)\big\}$.
\medskip

Can we generalize this to 
all holomorphic foliations on $\mathbb{P}^2_\mathbb{C}$ ? The answer is yes:

\begin{thm}\label{thm:main}
{\sl Let $\mathcal{F}$ be a holomorphic foliation on $\mathbb{P}^2_\mathbb{C}$. 
There exists a birational self map $\phi$ from $\mathbb{P}^2_\mathbb{C}$ 
into itself 
such that $\phi^{-1}\mathcal{F}$ has at most one singular point.}
\end{thm}

Note that many properties are preserved by conjugacy
by birational maps, for instance the existence of 
invariant algebraic curves of genus $\geq 2$, but
also the existence of dense leaves.

\begin{cor}\label{cor:main}
{\sl Let $\mathcal{F}$ be an algebraic foliation on $\mathbb{C}^2$.
There exists a birational map $\phi$ from $\mathbb{C}^2$ into itself
such that $\phi^{-1}\mathcal{F}$ has no singular point.}
\end{cor}

In a certain sense Corollary \ref{cor:main} indicates 
that algebraic foliations of $\mathbb{C}^2$ have no special 
properties.

Note that the situation is completely different on $\mathbb{P}^1_\mathbb{C}\times\mathbb{P}^1_\mathbb{C}$; in that space there exist foliations without singularities: {\sl a 
foliation on $\mathbb{P}^1_\mathbb{C}\times\mathbb{P}^1_\mathbb{C}$
birationally conjugate to a foliation without singularities
has a rational first integral} (\cite{Brunella}).

\bigskip

To prove Theorem \ref{thm:main} we will use the following 
statement:

\begin{lem}\label{lem:tec}
Let $\mathcal{F}$ be a holomorphic codimension one foliation
on $\mathbb{P}^2_\mathbb{C}$. Assume that $\mathcal{F}$
has $N$ singular points, and among them $n\geq 2$ singular points on 
a line $L$. There exists a birational self 
map $\phi$ from $\mathbb{P}^2_\mathbb{C}$ 
into itself such that $\phi^{-1}\mathcal{F}$ has 
$N-n+1$ singular points, one on $L$ and $N-n$ on 
$\mathbb{P}^2_\mathbb{C}\smallsetminus L$.
\end{lem}

\begin{proof}
\begin{itemize}
\item[$\diamond$] Assume first that $L$ is not invariant
by $\mathcal{F}$. Let us choose coordinates such that 
\[
\left\{
\begin{array}{lll}
L=(z=0)\\
(0:1:0)\not\in\mathrm{Sing}(\mathcal{F})\\
\text{$L$ is not tangent to the leaf of $\mathcal{F}$ through $(0:1:0)$ at $(0:1:0)$ \quad $(\star)$} 
\end{array}
\right.
\]
The foliation $\mathcal{F}$ is defined by 
\[
\omega=A(x,y,z)\,\mathrm{d}x+B(x,y,z)\,\mathrm{d}y+C(x,y,z)\,\mathrm{d}z.
\]
Let us recall the Euler condition: $xA+yB+zC=0$; in particular
note that $B(0,1,0)=0$. The tangency condition $(\star)$
says that $A(0,1,0)\not=0$.

Let us consider the birational involution $\phi$ of 
$\mathbb{P}^2_\mathbb{C}$ into itself given by 
\[
\phi\colon(x:y:z)\dashrightarrow(xz:-yz+x^2:z^2);
\]
note that 
\begin{align*}
    & \mathrm{Ind}(\phi)=\{(0:1:0)\}, 
    && \mathrm{Exc}(\phi)=\{(z=0)\}
\end{align*}
where $\mathrm{Ind}(\phi)$ (resp. $\mathrm{Exc}(\phi)$) denotes 
the set of indeterminacy (resp. the exceptional set) of $\phi$.
In particular, $\phi$ defines an isomorphism of  $\{z\not=0\}$.
One has 
\begin{eqnarray*}
\phi^*\omega&=&A(xz,-yz+x^2,z^2)(x\,\mathrm{d}z+z\,\mathrm{d}x)\\
& & \hspace{1cm}+B(xz,-yz+x^2,z^2)(-y\,\mathrm{d}z-z\,\mathrm{d}y+2x\,\mathrm{d}x)+2C(xz,-yz+x^2,z^2)z\,\mathrm{d}z,
\end{eqnarray*}
and 
\[
\phi^*\omega\big\vert_{z=0}=xA(0,x^2,0)\,\mathrm{d}x+\underbrace{B(0,x^2,0)}_{=0}\big(-y\,\mathrm{d}z+2x\,\mathrm{d}x\big)=xA(0,x^2,0)\,\mathrm{d}x;
\]
in particular $\phi^*\omega\big\vert_{z=0}$ vanishes 
only at $x=0$ since $A(0,1,0)\not=0$. As a consequence, 
$\#\big(\mathrm{Sing}(\phi^{-1}\mathcal{F})\cap(z=0)\big)=1$
and 
$\#\,\mathrm{Sing}(\phi^{-1}\mathcal{F}\big)=N-n+1$.

\item[$\diamond$] Suppose now that $L$ is invariant by 
$\mathcal{F}$. Let us choose coordinates such that
$\left\{
\begin{array}{lll}
L=(z=0)\\
(0:1:0)\not\in\mathrm{Sing}(\mathcal{F})
\end{array}
\right.$
The foliation $\mathcal{F}$ is defined by 
\[
\omega=zA(x,y,z)\,\mathrm{d}x+zB(x,y,z)\,\mathrm{d}y+C(x,y,z)\,\mathrm{d}z.
\]
Let us still consider the birational involution $\phi$ of 
$\mathbb{P}^2_\mathbb{C}$ into itself given by 
\[
\phi\colon(x:y:z)\dashrightarrow(xz:-yz+x^2:z^2)
\]
that defines an isomorphism of $\{z\not=0\}$.
Then
\begin{eqnarray*}
\phi^*\omega&=&z\Big(zA(xz,-yz+x^2,z^2)(x\,\mathrm{d}z+z\,\mathrm{d}x)\\
& &\hspace{1cm}+zB(xz,-yz+x^2,z^2)(-y\,\mathrm{d}z-z\,\mathrm{d}y+2x\,\mathrm{d}x)+2C(xz,-yz+x^2,z^2)\,\mathrm{d}z\Big),
\end{eqnarray*}
and the restriction of $\phi^{-1}\mathcal{F}$ on $(z=0)$ is described by
the $1$-form $2C(0,x^2,0)\,\mathrm{d}z$.
But $C(0,1,0)\not=0$ (because $(0:1:0)\not\in\mathrm{Sing}(\mathcal{F})$) 
so $C(0,x^2,0)=0$ if and only if $x=0$. In particular, 
$\#\big(\mathrm{Sing}(\phi^{-1}\mathcal{F})\cap(z=0)\big)=1$
and 
$\#\,\mathrm{Sing}(\phi^{-1}\mathcal{F}\big)=N-n+1$.
\end{itemize}
\end{proof}
 
\begin{proof}[Proof of Theorem \ref{thm:main}]
We get the result by iteration using Lemma \ref{lem:tec}.
\end{proof}

Assume that $\mathcal{F}$ is an algebraic foliation
of any codimension of $\mathbb{C}^n$; is it true 
that there exists a birational self map of 
$\mathbb{P}^n_\mathbb{C}$ into itself such that
$\phi^{-1}\mathcal{F}$ has no singularity ?
In the real case (for example in 
$\mathbb{P}^2_\mathbb{R}$) what is the best result
we can expect ?

Remark that following \cite{CDGBM} in which 
we classify, up to automorphisms of 
$\mathbb{P}^2_\mathbb{C}$, the quadratic foliations having 
a unique singularity there is a lot of activity around the 
foliations of $\mathbb{P}^2_\mathbb{C}$ having a unique 
singularity (\emph{for instance} \cite{FernandezPuchuriRosas, Alcantara, Coutinho, CoutinhoJales, CoutinhoFerreira2, CoutinhoFerreira}).

\vspace*{2cm}

\bibliographystyle{alpha}
\bibliography{biblio}

\begin{thebibliography}{CDGBM10}

\bibitem[Alc22]{Alcantara}
C.~R. Alc\'{a}ntara.
\newblock Special foliations on {$\Bbb {CP}^2$} with a unique singular point.
\newblock {\em Res. Math. Sci.}, 9(1):Paper No. 15, 11, 2022.

\bibitem[Bru15]{Brunella}
M.~Brunella.
\newblock {\em Birational geometry of foliations}, volume~1 of {\em IMPA
  Monographs}.
\newblock Springer, Cham, 2015.

\bibitem[CDGBM10]{CDGBM}
D.~Cerveau, J.~D\'{e}serti, D.~Garba~Belko, and R.~Meziani.
\newblock G\'{e}om\'{e}trie classique de certains feuilletages de degr\'{e}
  deux.
\newblock {\em Bull. Braz. Math. Soc. (N.S.)}, 41(2):161--198, 2010.

\bibitem[CF20]{CoutinhoFerreira}
S.~C. Coutinho and F.~R. Ferreira.
\newblock A family of foliations with one singularity.
\newblock {\em Bull. Braz. Math. Soc. (N.S.)}, 51(4):957--974, 2020.

\bibitem[CF21]{CoutinhoFerreira2}
S.~C. Coutinho and F.~R. Ferreira.
\newblock Correction to: {A} family of foliations with one singularity.
\newblock {\em Bull. Braz. Math. Soc. (N.S.)}, 52(3):767--769, 2021.

\bibitem[CJ21]{CoutinhoJales}
S.~C. Coutinho and L.~F.~G. Jales.
\newblock Foliations with one singularity and finite isotropy group.
\newblock {\em Bull. Sci. Math.}, 169:Paper No. 102988, 18, 2021.

\bibitem[Cou22]{Coutinho}
S.~C. Coutinho.
\newblock On the classification of foliations of degree three with one
  singularity.
\newblock {\em J. Symbolic Comput.}, 112:62--78, 2022.

\bibitem[FPR22]{FernandezPuchuriRosas}
P.~Fern\'{a}ndez, L.~Puchuri, and R.~Rosas.
\newblock Foliations on {$\Bbb{P}^2$} with only one singular point.
\newblock {\em Geom. Dedicata}, 216(5):Paper No. 59, 17, 2022.

\end{thebibliography}

\nocite{}

\end{document}